\documentclass[final,letterpaper]{article}
\usepackage{times}
\usepackage{amsthm}
\usepackage{amsmath}
\usepackage{amssymb}
\usepackage{xspace}
\usepackage{xcolor}

\newtheorem{theorem}{Theorem}

\newtheorem{lemma}[theorem]{Lemma}
\newtheorem{corollary}[theorem]{Corollary}
\theoremstyle{definition}
\newtheorem{assumption}[theorem]{Assumption}
\newtheorem{remark}[theorem]{Remark}
\newtheorem{example}[theorem]{Example}

\title{Variational regularization with oversmoothing penalty term in Banach spaces}

\author{Robert Plato\footnotemark[1]
	 \and Bernd Hofmann\footnotemark[2]}

\newcommand{\para}{\alpha}

\newcommand{\parb}{\beta}
\newcommand{\pardel}{\para_*}

\newcommand{\ix}{\mathcal{X}}
\newcommand{\yps}{\mathcal{Y}}
\newcommand{\pp}{p}
\newcommand{\qq}{q}

\newcommand{\fdelta}{f^\delta}

\newcommand{\Landau}{\mathcal{O}}
\newcommand{\refeq}[1]{(\ref{eq:#1})}

\newcommand{\G}{G}

\newcommand{\norm}[1]{\Vert \hspace{0.4mm} #1 \hspace{0.4mm} \Vert}

\newenvironment{myenumerate}{%
\begin{list}{(\alph{enumcount})}
{\setcounter{enumcount}{1}\usecounter{enumcount}
\setlength{\topsep}{1mm}
\setlength{\itemsep}{0mm}
\setlength{\labelwidth}{0mm}
\setlength{\labelsep}{1mm}
\setlength{\itemindent}{1mm}
\setlength{\leftmargin}{0mm}
}}{\end{list}}

\newcommand{\DF}{\Dset(\F)}

\newcommand{\Dset}{\mathcal{D}}

\newcommand{\myu}{u}

\newcommand{\ust}{u^\dagger}

\newcommand{\fst}{f^\dagger}
\newcommand{\fdel}{f^\delta}
\newcommand{\fn}{f_n}
\newcommand{\un}{u_n}
\newcommand{\hn}{h_n}

\newcommand{\tikreg}{Tikhonov regularization\xspace}
\newcommand{\tifu}{Tikhonov functional\xspace}

\newcommand{\tinybullet}{{\tiny \raisebox{0.6mm}{$ \bullet $}}}
\newenvironment{mylist}{%
\begin{list}{\tinybullet}
{\setlength{\topsep}{0.2cm}
\setlength{\itemsep}{0mm}
\setlength{\labelwidth}{0mm}
\setlength{\labelsep}{3mm}
\setlength{\itemindent}{3mm}
\setlength{\leftmargin}{0mm}
}}{\end{list}}

\newenvironment{mylist_indent}{%
\begin{list}{\tinybullet}
{\setlength{\topsep}{0.2cm}
\setlength{\itemsep}{0mm}
\setlength{\labelwidth}{2mm}
\setlength{\labelsep}{3mm}
\setlength{\itemindent}{0mm}
\setlength{\leftmargin}{5mm}
}}{\end{list}}

\newcommand{\kla}[1]{(#1)}
\newcommand{\mfrac}[2]{\dfrac{\mbox{\footnotesize \raisebox{-0.5mm}{$#1$}}}%
{\mbox{\footnotesize \raisebox{0.8mm}{$#2$}}}}
\newcommand{\proofend}{\qquad \endproof}

\newcommand{\Landauno}[1]{\Landau\kla{#1}}
\newcommand{\as}{\quad \text{as} \ \ }

\newcommand{\R}{\mathcal{R}}

\newcommand{\remarkend}{\quad \ensuremath{\vartriangle}}

\newcommand{\defeq}{:=}
\newcommand{\for}{\quad \text{for} \ \ }

\newcommand{\ints}[4]{\mathop{\raisebox{-0.1mm}{\mbox{\Large$ \textstyle \int $}}}\nolimits_{\hspace{-1mm}#1}^{#2} \hspace{-0.2mm} #3 \, #4}

\newcommand{\rhs}{right-hand side\xspace}

\newenvironment{myenumerate_indent}{%
\begin{list}{(\alph{enumcount})}
{\setcounter{enumcount}{1}\usecounter{enumcount}
\setlength{\topsep}{1mm}
\setlength{\itemsep}{0mm}
\setlength{\labelwidth}{5mm}
\setlength{\labelsep}{2mm}
\setlength{\itemindent}{-0mm}
\setlength{\leftmargin}{7mm}
}}{\end{list}}

\newcommand{\uhat}{\widehat{u}}

\newcommand{\inset}[1]{\{ \, #1 \, \}}

\newcommand{\ualpaux}[1][\parb]{\widehat{u}_{#1}}

\newcommand{\fab}{r(a+1)}
\newcommand{\norma}[1]{\norm{#1}_{-a}}

\newcommand{\normone}[1]{\norm{#1}_{1}}
\newcommand{\normmone}[1]{\norm{#1}_{-1}}

\newcommand{\normyps}[1]{\norm{#1}}
\newcommand{\normypsqua}[1]{\norm{#1}^\myr}
\newcommand{\normix}[1]{\norm{#1}}

\newcommand{\lfrac}[2]{#1/#2}

\newcommand{\ca}{c_a}
\newcommand{\cb}{C_a}
\newcommand{\upardel}[1][\para]{{u}_{#1}^\delta}

\newcommand{\updd}{\upardel[\pardel]}

\newcommand{\udel}{\updd}
\newcommand{\udelb}{u^\delta}
\newcommand{\tikfun}{T}
\newcommand{\pad}[2][\para]{\tikfun_{#1}^\delta(#2)}
\newcommand{\padn}[2][\para]{\tikfun_{#1}^n(#2)}
\newcommand{\padpur}[1][\para]{\tikfun_{#1}^\delta}
\newcommand{\padnpur}[1][\para]{\tikfun_{#1}^n}
\newcommand{\Jdel}[1]{\norm{\myF{#1} - \fdel}}

\newcommand{\myomega}[1]{\normonequa{#1}}
\newcommand{\myomegab}[1]{\normonequa{#1-\ubar}}
\newcommand{\myomegabp}[1]{\normone{#1-\ubar}}
\newcommand{\F}[1]{\Fpur#1}
\newcommand{\Fpur}{F}

\newcommand{\ubar}{\overline{u}}

\newcommand{\mainassump}{Let Assumption \ref{th:main_assump} be satisfied.\xspace}
\newenvironment{myenumerate_roman}{%
\begin{list}{(\roman{enumcountroman})}
{\setcounter{enumcountroman}{1}\usecounter{enumcountroman}
\setlength{\topsep}{0.2cm}
\setlength{\itemsep}{0mm}
\setlength{\labelwidth}{0mm}
\setlength{\labelsep}{3mm}
\setlength{\itemindent}{3mm}
\setlength{\leftmargin}{0mm}
}}{\end{list}}

\newcommand{\wrt}{with respect to\xspace}

\newcommand{\bh}[1]{#1}

\newcommand{\normonequa}[1]{\norm{#1}_{1}^\myr}
\newcommand{\myr}{r}
\newcommand{\myrr}{\tau}

\newcommand{\Jdelqua}[1]{\norm{F(#1) - \fdel}^\myr}
\newcommand{\cp}{c_p}
\newcommand{\fracpower}{fractional power\xspace}
\newcommand{\fracpowers}{\fracpower{}s\xspace}
\newcommand{\domain}{\mathcal{D}}
\newcommand{\range}{\mathcal{R}}
\newcommand{\mykap}{\kappa}
\newcommand{\RG}{\overline{\R(\G)}}
\newcommand{\linf}{L^\infty(0,1)}
\newcommand{\lone}{L^1(0,1)}
\newcommand{\myF}[1]{F(#1)}
\newcommand{\normmax}[1]{\norm{#1}_{\infty}}
\newcommand{\weakstararrow}{\rightharpoonup^\ast}
\newcommand{\mysl}{\mathcal{S}}
\newcommand{\minpad}[1][\para]{\tikfun_{#1,\ast}^\delta}
\newcommand{\er}{e_r}
\newcommand{\myseq}[1]{\{#1\}}
\newcommand{\ixtwo}{\overline{\R(G)}}
\newcommand{\postype}{non-negative type\xspace}
\newcommand{\Postype}{Non-negative type\xspace}
\newcommand{\loginv}[1]{\kla{\log#1}^{-1}}

\begin{document}
\date{}
\maketitle

\noindent
{\large
Dedicated to our distinguished colleague M. Thamban Nair on the occasion of his 65th birthday}

\renewcommand{\thefootnote}{\fnsymbol{footnote}}
\footnotetext[1]{Department of Mathematics, University of Siegen,
Walter-Flex-Str.~3, 57068 Siegen, Germany}
\footnotetext[2]{Faculty of Mathematics, Chemnitz University of Technology, 09107 Chemnitz, Germany.}

\newcounter{enumcount}
\newcounter{enumcountroman}
\renewcommand{\theenumcount}{(\alph{enumcount})}
\bibliographystyle{plain}
\begin{abstract}
In the present work, we discuss \bh{variational regularization for ill-posed nonlinear problems} with focus on an oversmoothing penalty term. This means in our model that
the searched-for solution of the considered nonlinear operator equation does not belong to the domain of definition of the penalty functional. In the past years, such variational regularization has been investigated comprehensively in Hilbert scales. Our present study tries to continue and to extend those investigations to Banach scales. This new study includes convergence rates results
for a priori choices of the regularization parameter,
both for H\"older-type smoothness and low order-type smoothness.
The necessary tools for low order smoothness in the Banach space setting are provided.
\end{abstract}

%


%
\section{Introduction}
\label{intro}
The subject of this paper are nonlinear operator equations of the form
\begin{equation} \label{eq:opeq}
 \F(u) = \fst \,,
\end{equation}
where $ \F: \ix \supset \DF \to \yps $ is a nonlinear operator
between infinite-dimensional Banach spaces $ \ix $ and $ \yps $ with norms $\|\cdot\|$.
We suppose that the \rhs $ \fst \in Y $ is approximately given as $ \fdelta \in \yps$ satisfying the deterministic noise model
\begin{equation} \label{eq:noise}
 \normyps{ \fdelta-\fst } \le \delta,
\end{equation}
with the noise level $\delta \ge 0$.
Throughout the paper, it is assumed that the considered equation~\eqref{eq:opeq} has a solution $ \ust \in \DF $ and is, at least at $ \ust$, locally ill-posed in the sense of \cite[Def.~3]{HofPla18}.

For finding stable approximations to the solution  $ \ust \in \DF $ of equation \eqref{eq:opeq}, we exploit a variant of variational regularization with regularization parameter $\para>0$, where the regularized solutions $\upardel$ are minimizers of the extremal problem
\begin{equation} \label{eq:TR}
\pad{\myu} \defeq \Jdelqua{\myu} + \para \myomega{\myu-\ubar} \to \min \quad \textup{subject to} \quad  \myu \in \DF,
\end{equation}
with some exponent $ r > 0 $ being fixed.
In addition, $\ubar \in \ix_1$ occurring in the penalty term of the Tikhonov functional $T_\alpha^\delta$ denotes an initial guess. In this context, $ \normone{\cdot} $ is a norm on a densely defined subspace $\ix_1$ of $\ix$, which is stronger than the original norm $ \normix{\cdot} $ in $ \ix $.
Note that we \bh{restrict our consideration here to identical exponents for the misfit term and the penalty functional. This restriction is actually only for technical reasons.}

Precisely, we define the stronger norm $\normone{\cdot}$ by
$ \normone{u} = \norm{G^{-1}u}, u \in \R(G) $, where the generator
$ G: \ix \to \ix $ with range $\R(G)$ is a bounded linear operator, which is one-to-one and has an unbounded inverse $ G^{-1} $. Further conditions on the operator $ G $ are given in Section
\ref{postype_operators} below.

In the present work, we discuss the nonlinear Tikhonov-type regularization \eqref{eq:TR} with focus on an oversmoothing penalty term. This means in our model that we have $ \ust \not \in \ix_1 $, or in other words $\|\ust\|_1=+\infty,$ which is an expression of
`non-smoothness' of the solution $\ust$ with respect to the reference Banach space $\ix_1$. Variational regularization of the form \eqref{eq:TR} with $r=2$ and oversmoothing penalty for nonlinear ill-posed operator equations \eqref{eq:opeq} has been investigated comprehensively in the past four years in Hilbert scales, and we refer to \cite{Hofmann_Mathe[18],HofPla20} as well as further to the papers \cite{GHH20,HofHof20,HHMP22,HofMat20}. For related results on linear problems, see, e.g.,
\cite{Natterer[77]} and more recently \cite{George[08],Nair_Pereverzev_Tautenhahn[05]}.
Our present study continues and extends, along the lines of \cite{HofPla20}, the investigations on nonlinear problems to Banach scales. This new study includes a fundamental error estimate, cf.
\refeq{fin} below, and
convergence results as well as convergence rates results
for a priori choices of the regularization parameter,
both for H\"older-type smoothness and low order smoothness.
The necessary tools for low order smoothness in the Banach space setting are provided.
In addition, a relaxed nonlinearity and smoothing condition on the operator $F$
is considered that turns out to be useful for maximum norms.

Banach space results for the discrepancy principle in a pure equation form have already been proven for the oversmoothing case in the recent paper \cite{CHY21}. In parallel, such results have been developed for oversmoothing subcases to variants of $\ell^1$-regularization and sparsity promoting wavelet regularization in \cite[Sec.~5]{Miller21} and \cite[Chap.~5]{MillerDiss22}.

The outline of the remainder is as follows: in Section~\ref{sec:preparations} we summarize prerequisites and assumptions for the main results in the sense of error estimates and convergence rates for the regularized solutions. These main results will then be presented in Section~\ref{sec:apriori_parameter_choices}.
An illustrative example to illuminate the general theory is given in Section~\ref{sec:example}.
The final Section~\ref{sec:verifications} contains the proofs of the main results, which in particular need the adapted construction of `smooth' auxiliary elements that approximate the `non-smooth' solution sufficiently well.
%

%
\section{Prerequisites and assumptions}
\label{sec:preparations}
In this section, some preparations are carried out. We introduce a scale of Banach spaces generated by an operator of positive type, introduce the logarithm of a positive operator and formulate the basic assumptions for this paper.
Moreover, we discuss well-posedness and stability assertions for the variant of variational regularization under consideration.

\subsection{\Postype operators, fractional powers, and regularization operators}
\label{postype_operators}
Let $ \ix $ be a Banach space and $ G: \ix \to \ix $ be a bounded linear operator of \postype, i.e.,
\begin{align}
& G + \parb I: \ix \to \ix \ \textup{one-to-one and onto},
\qquad \norm{(G + \parb I)^{-1}} \le \frac{\kappa_*}{\parb}, \quad \parb > 0,
\label{eq:postype}
\end{align}
for some finite constant $ \kappa_* > 0 $.
Fractionals powers of \postype operators may be defined as follows
\cite{Balakrishnan[59],Balakrishnan[60]}:
\begin{myenumerate_indent}
\item
For $ 0 < \pp < 1 $, the fractional power $ G^\pp : \ix \to \ix $ is defined by
\begin{align}
G^\pp u \defeq \mfrac{\sin \pi \pp }{\pi}
\ints{0}{\infty}{s^{\pp-1} (G + sI)^{-1} G u }{ ds}
\for u \in \ix.
\label{eq:frac-power}
\end{align}
\item
For arbitrary values $ \pp > 0 $,
the bounded linear operator $  G^\pp : \ix \to \ix $ is defined by
$$
G^\pp  \defeq G^{\pp - \lfloor \pp \rfloor} G^{\lfloor \pp \rfloor}.
$$
We moreover use the notation $ G^0 = I $.
\end{myenumerate_indent}
In what follows, we shall need
the interpolation inequality for \fracpowers of operators, see, e.g.,
\cite{Komatsu[66]} or
\cite[Proposition 6.6.4]{Haase[06]}:
for each pair of real numbers \linebreak $ 0 < \pp < \qq $,
there exists some finite constant $ c =c(\pp,\qq) > 0 $ such that
\begin{align}
\norm{ \G^\pp u }
\le c \norm{\G^\qq u}^{\pp/\qq} \norm{ u }^{1-\pp/\qq}
\for u \in \ix.
\label{eq:interpol2}
\end{align}
For $ 0 < \pp < 1 = \qq $, the value of the constant can be chosen as follows, $ c = 2(\kappa_* + 1)$, cf., e.g., \cite[Corollary 1.1.19]{Plato[95]}.
Throughout the paper, we assume that the operator $ \G $ is one-to-one
and that the inverse $ \G^{-1} $ is an unbounded operator.
Then for each $\pp > 0 $, the fractional power
$ G^{\pp} $ is also one-to-one, and we use the notation
$ G^{-\pp} = (G^\pp)^{-1} $.
We do not assume that the operator $ G $ has dense range in $ \ix $.

The scale of normed spaces $\{\ix_\tau\}_{\tau \in \mathbb{R}}$,
generated by $\G$, is given by the formulas
\begin{align}
& \ix_\tau=\range(\G^\tau) \ \textup{ for }\, \tau > 0, \qquad
\ix _\tau= \ix \ \textup{ for }\, \tau \le 0, \nonumber \\
& \|u\|_\tau:=\|\G^{-\tau} u\| \ \textup{ for }\, \tau \in \mathbb{R}, \ u \in \ix_\tau.
\label{eq:taunorm}
\end{align}
For $ \tau < 0 $, topological completion of the spaces
$X_\tau = \ix$ with respect to the norm $ \|\cdot\|_\tau $ is not needed in our setting.
We note that $ (\G^p)_{p\ge 0} $ defines a $ C_0 $-semigroup on
$ \ixtwo $,
which in particular means that $ \G^p \myu \to \myu $ for $ p \downarrow 0 $ is valid for any
$ u \in \overline{\R(G)} $ (cf.~\cite[Proposition 3.1.15]{Haase[06]}). Finally, we note that
\begin{equation} \label{eq:chain1}
\R(G^{\tau_2}) \subset\R(G^{\tau_1}) \subset \overline{\R(\G)} \quad \mbox{for all} \;0 < \tau_1 < \tau_2 < \infty.
\end{equation}
\subsection{The logarithm $ \mathbf{\log G } $}
For the consideration of low order smoothness, we need to introduce the logarithm of
$ \G $. For selfadjoint operators in Hilbert spaces this can be done by spectral analysis, and we refer in this context for example to \cite{Hohage[00],Mahale_Nair[07]}.
In Banach spaces, $ \log G $ may be defined as the infinitesimal generator of the $ C_0 $-semigroup $ (\G^p)_{p \ge 0} $ considered on $ \overline{\R(G)} $:
\begin{align*}
(\log \G) \myu & = \lim_{p \downarrow 0} \tfrac{1}{p}\kla{\G^p \myu-\myu},
\quad \myu  \in \domain(\log \G),
\end{align*}
where
\begin{align*}
\domain(\log\G) & =
\inset{ \myu \in \ix : \lim_{p \downarrow 0} \tfrac{1}{p}\kla{\G^p \myu- \myu}
\ \textup{exists} },
\end{align*}
cf.,~e.g.,~\cite{Nollau[69]} or \cite[Proposition 3.5.3]{Haase[06]}.
Low order smoothness of an element $ \myu \in \ix $ by definition then means
$ \myu \in \domain(\log \G) $. Note that
we obviously have $ \domain(\log \G) \subset \RG $. In addition, $ \R(\G^p) \subset \domain(\log \G)$ is valid for arbitrarily small $p>0$,
which follows from \cite[Satz 1]{Nollau[69]}. Summarizing the above notes, we have a chain of subsets of $\ix$ as
\begin{equation} \label{eq:chain2}
\R(\G^p) \subset \domain(\log \G)  \subset \overline{\R(\G)} \quad \mbox{for all} \;p>0.
\end{equation}
This means that also in the Banach space setting, any H\"older-type smoothness is stronger than low order smoothness.
\subsection{Main assumptions}
In the following assumption, we briefly summarize the structural properties of the operator $F$ and of its domain $\DF$, in particular with respect to the solution $\ust$ of the operator equation
\eqref{eq:TR}.
\begin{assumption}
\label{th:main_assump}
\begin{myenumerate_indent}
\item \label{it:conda}
The operator $ \F: \ix \supset \DF \to \yps $ is
continuous \wrt the norm topologies
of the spaces $ \ix $ and $ \yps $.

\item
The domain of definition $ \DF \subset \ix $ is a closed subset of $\ix$.
\item
Let $\Dset \defeq \DF \cap X_1 \neq \varnothing $.

\item
Let the solution $ \ust \in \DF $ to equation \eqref{eq:opeq} with \rhs $\fst$ be
an interior point of the domain $ \DF $.

\item
Let the data $ \fdelta \in \yps$ satisfy the noise model \eqref{eq:noise}, and let the
initial guess $\ubar$ satisfy $ \ubar \in \ix_1 $.

\item
\label{it:normequiv_b}
Let $ a > 0 $, and let there exist finite constants $ 0 < \ca \le \cb $
and $ c_0, c_1 > 0 $ such that
the following holds:
\begin{mylist}
\item For each $ u \in \Dset $ satisfying $ \norma{u-\ust} \le c_0 $, we have
\begin{align}
\label{eq:normequiv_a}
\normyps{\myF{\myu} - \fst} \le \cb\norma{\myu - \ust}.
\end{align}

\item
For each $ u \in \Dset $ satisfying $ \normyps{\myF{\myu} - \fst} \le c_1 $, we have
\begin{align}
\label{eq:normequiv_b}
\ca\norma{\myu - \ust} \le \normyps{\myF{\myu} - \fst}.
\end{align}
\end{mylist}
\item
\label{it:G-compact}
The operator $ G: \ix \to \ix$ is compact.

\item
\label{it:G-preadjoint}
The Banach space $ \ix $ has a separable predual space $\tilde \ix$ with $\tilde \ix^*=\ix$. Moreover, the operator $ G $ has a bounded preadjoint operator $\tilde G:\tilde \ix \to  \tilde \ix$, i.e., $\tilde G^*=G$.
\end{myenumerate_indent}
\end{assumption}
\begin{remark} \label{th:rem-assump}
Items \ref{it:G-compact} and \ref{it:G-preadjoint} are needed for the proof our version of well-posedness (existence and stability of regularized solutions in the norm of $ \ix $) in the sense of Theorem \ref{th:tifu-well-posed} below. Note that the preadjoint operator $\tilde G $ is necessarily also compact, see, e.g.,
\cite{Yosida[80]}.
\end{remark}
\begin{remark} \label{rem:rem1}
From the inequality \eqref{eq:normequiv_b} of item \ref{it:normequiv_b} in Assumption~\ref{th:main_assump}, we have for $\ust \in X_1$ that $\ust$ is the uniquely determined solution to equation \eqref{eq:opeq} in the set $\Dset$. For $\ust \notin X_1$, there is no solution at all to \eqref{eq:opeq} in $\Dset$. But in both cases, alternative solutions $u^* \notin X_1$ with $u^* \in \DF$ and $F u^*=\fst$ cannot be excluded in general. However, there is an exception if $u^*$ is an interior point of $\DF$. Then a solution  $u^* \in \overline{\R(G)}$ to equation \eqref{eq:opeq} with right-hand side $\fst$ satisfies with $u=u^*$ the inequality \eqref{eq:normequiv_b}. This is a
consequence of the continuity of $F$ from item \ref{it:conda} of Assumption~\ref{th:main_assump}.
\end{remark}

\subsection{Existence and stability of regularized solutions}
Recall that, for $\alpha>0$, minimizers of the Tikhonov functional $\padpur $ introduced in  \eqref{eq:TR}
are denoted by $ \upardel$, i.e., we have
$$\pad{\upardel} = \min_{\myu \in \DF} \ \pad{u}. $$
Evidently, by definition of the penalty term,
$ \upardel \in \Dset$ holds.

The extremal problem \eqref{eq:TR} for finding regularized solutions $ \upardel \in \Dset $ is well-posed  in a sense specified in the following Theorem~\ref{th:tifu-well-posed}. Precisely,
both existence of minimizers and stability with respect to data perturbations can be guaranteed for all $\alpha>0$ and all $\fdelta \in \yps$.

\begin{theorem}
\label{th:tifu-well-posed}
\mainassump
Then the following holds:
\begin{myenumerate_indent}
\item
For all $\alpha>0$ and all $\fdelta \in \yps$ there exists a minimizer $ \upardel$ of the \tifu $\padpur$, which belongs to the set $\Dset $.

\item
For $\alpha>0$,
each minimizing sequence of $\padpur$ over $\Dset$ has a subsequence that converges in the norm of $ \ix $ to a minimizer $ \upardel \in \Dset $ of the Tikhonov functional.

\item
For $\alpha>0$, the regularized solutions $ \upardel$ are stable in the norm of $ \ix $ with respect to small perturbations in the data $\fdelta\in \yps$.
\end{myenumerate_indent}
\end{theorem}
The proofs of the above theorem and of the majority of subsequent results in Section~\ref{sec:apriori_parameter_choices}
are postponed to Section \ref{sec:verifications}.
Here we only note that the proof of Theorem \ref{th:tifu-well-posed}, more or less, rely on standard techniques and results related with the regularization of ill-posed minimization.
A special feature of the theorem, however, is stability with respect to the given norm on the space $ \ix $, not only with respect to the corresponding weak topology.
This is due to the fact that the norm in the penalty term is generated by an unbounded operator which has a compact inverse.
\begin{remark}
We note that the minimizer of the \tifu may be non-unique for nonlinear forward operators $F$,
because $T_\alpha^\delta$ can be a non-convex functional as a consequence of a non-convex misfit term
$\|\myF{\myu} - f^\delta\|^\myr$. This is, for example, the case when $\myF{u}:=u \star u$ represents the autoconvolution
operator with $\ix=\yps=\DF:=L^2(0,1)$ (cf., e.g., \cite{Buerger15}).
Then we have for $\ubar=0$ that
$T_\alpha^\delta(u)=T_\alpha^\delta(-u)$,
which illustrates the non-uniqueness phenomenon.
\end{remark}
It belongs to the main goals of this study to verify error estimates and derive convergence rates results for the variant \eqref{eq:TR} of variational regularization with oversmoothing penalty, where $ \ust \not \in \ix_1 $.
\section{Error estimate and a priori parameter choices}
\label{sec:apriori_parameter_choices}
We start with an error estimate result that provides the basis for the analysis of the regularizing properties, including convergence rates under a priori parameter choices. In what follows, we use the notation
\begin{align*}
\mykap \defeq \frac{1}{\fab}.
\end{align*}
\begin{theorem}
\label{th:upardel-esti}
\mainassump
Then there exist finite positive constants $ K_1, \para_0 $ and $ \delta_0 $ such that for $ 0 < \para \le \para_0 $ and $ 0 < \delta \le \delta_0 $,
an error estimate for the regularized solutions as
\begin{align}
\label{eq:fin}
\norm{\upardel -\ust} \le f_1(\para) + K_1 \frac{\delta}{\para^{\mykap a}}
\end{align}
holds, where $f_1(\para)$ for $ 0 < \para \le \para_0 $ is some bounded function satisfying:
\begin{mylist_indent}
\item (No explicit smoothness)
If $ \ust \in \overline{\R(G)} $, then $ f_1(\para) \to 0 $ as $ \para \to 0 $.
\item (H\"older smoothness)
If $ \ust \in \ix_\pp $ for some $ 0 < \pp \le 1 $, then
$ f_1(\para) = \Landauno{\para^{\mykap p}} $ as $ \para \to 0 $.
\item (Low order smoothness)
If $ \ust \in \domain(\log \G) $, then
$ f_1(\para) = \Landauno{\loginv{\frac{1}{\para}}} $ as $ \para \to 0 $.
\end{mylist_indent}
\end{theorem}
Theorem \ref{th:upardel-esti} allows us to derive regularizing properties
of variational regularization with oversmoothing penalty.
This will be the topic of subsequent considerations aimed at obtaining convergence and rates results for appropriate a priori parameter choices,
which culminate in Theorem~\ref{th:apriori}.
For evaluating the strength of smoothness for the three different occurring situations in Theorem~\ref{th:upardel-esti}
(no explicit smoothness, H\"older smoothness and low order smoothness) we recall the chain \eqref{eq:chain2} of range conditions.

The following main theorem is a direct consequence of Theorem~\ref{th:upardel-esti}, because its proof is immediately based on the error estimate \eqref{eq:fin} with the respective properties
of the function $f_1(\para)$.
\begin{theorem}
\label{th:apriori}
\mainassump
\begin{mylist}
\item (No explicit smoothness)
Let $ \ust \in \overline{\R(G)} $. Then for any a priori parameter choice
$\pardel=\para(\delta)$
satisfying
$ \pardel \to 0 $ and $ \tfrac{\delta}{\pardel^{\mykap a}} \to 0 $ as $ \delta \to 0 $,
we have
$$ \normix{u_{\pardel}^\delta-\ust} \to 0 \quad \textup{ as } \delta \to 0.
$$

\item (H\"older smoothness)
Let $ \ust \in \ix_p $ for some $ 0 < p \le 1 $.
Then for any a priori parameter choice satisfying
$ \para_*=\para(\delta) \sim  \delta^{1/(\kappa(p+a))} $
we have
\begin{align*}
\normix{\udel -\ust} = \Landauno{\delta^{p/(p+a)}} \as \delta \to 0.
\end{align*}
\item (Low order smoothness)
Let $ \ust \in \domain(\log G) $.
Then for any a priori parameter choice satisfying
$ \para_*=\para(\delta) \sim \delta $, we have
\begin{align*}
\normix{\udel -\ust} = \Landauno{\loginv{\tfrac{1}{\delta}}} \as \delta \to 0.
\end{align*}
\end{mylist}
\end{theorem}
\section{An illustrative example}
\label{sec:example}
In what follows, we present an example with specific Banach spaces and nonlinear forward operator, which shows that the general mathematical framework developed in this paper is applicable. The considered basis space is $ \ix = \linf $
with the essential supremum norm $ \norm{\cdot} = \normmax{\cdot} $ possessing a separable predual space $\tilde \ix=L^1(0,1)$. The generator $G$  of the scale of normed spaces is given by
\begin{align*}
[Gu](x) = \int_0^x u(\xi) \, d\xi \qquad (0 \le x \le 1, \quad u \in \linf).
\end{align*}
Below we give some properties of $ G $:
\begin{mylist_indent}
\item
The operator $ G: \linf \to \linf $ is of \postype
with constant $ \kappa_* = 2 $,
see, e.g., \cite{Plato[95]}.

\item $G$ has a trivial nullspace and a non-dense range
$$ \R(G) = W^{1,\infty}_0(0,1) :=  \inset{ u \in W^{1,\infty}(0,1) : u(0) = 0 }, $$
with $$ \overline{\R(G)} =
C_0[0,1] := \inset{u \in C[0,1]:u(0) = 0 }.$$

\item
$G$ is a compact operator, which follows immediately from the Arzel\'{a}--Ascoli theorem.
\item
$G$ has a compact preadjoint operator $ \tilde G: \lone \to \lone $, which is characterized by $$ [\tilde Gv](x) = \int_x^1 v(\xi) \, d\xi \qquad (0 \le x \le 1, \quad v \in \lone).$$
\end{mylist_indent}
The nonlinear forward operator of this example is
$ F: \linf \to \linf $ given by
\begin{align*}
[F(u)](x) = \exp((Gu)(x)) \qquad (0 \le x\le 1,\quad u \in \linf).
\end{align*}
This operator $ F $ is Fr\'{e}chet differentiable on its domain of definition $ \domain(F) = \linf $, with
$ [F^\prime(u)] h=[F(u)] \cdot Gh $.
Now consider some function $ \ust \in \linf $ which is assumed to be fixed throughout this section. We then have
\begin{align*}
c_1 \le F\ust \le c_2 \ \textup{ on } [0,1], \; \mbox{with} \;
c_1 := \exp(-\normmax{G\ust}) > 0, \ c_2 \defeq \exp(\normmax{G\ust}),
\end{align*}
so that
\begin{align}
c_1 \vert Gh \vert \le  \vert F^\prime(\ust)h \vert \le c_2 \vert Gh \vert \quad \textup{on } [0,1] \qquad (h \in \linf).
\label{eq:Fprimebounds}
\end{align}
For any $ u \in \linf $, we denote by $ \Delta = \Delta(u) $ and
$ \theta = \theta(u) $ the following functions:
\begin{align*}
\Delta & \defeq Fu-F\ust \in \linf, \qquad \theta \defeq G(u-\ust) \in \linf.
\end{align*}
Thus,
$ \normmone{u-\ust} = \normmax{\theta} $, and we refer to \refeq{taunorm} for the definition of
$ \normmone{\cdot}$.

Below we show that the basic estimates \refeq{normequiv_a} and
\refeq{normequiv_b} are satisfied for that example with $a=1$. As a preparation, we note that
\begin{align}
\vert \Delta-  F^\prime(\ust) (u-\ust) \vert
\le \vert \theta \vert \, \vert \Delta \vert \quad \textup{on } [0,1],
\label{eq:hofmann-estimate}
\end{align}
and refer in this context to \cite[Sect.~4.4]{HHMP22}. \bh{In this reference. the  same $F$ is analyzed as an operator mapping in $L^2(0,1)$, where moreover
its relation to a parameter estimation problem
for an initial value problem of a first order ordinary differential equation
is outlined.}
\begin{myenumerate}
\item
We first show that \refeq{normequiv_a} holds. Even more general we show that it holds for any $ u \in \linf $ sufficiently close to $ \ust $, not only for $ u \in \ix_1$.
From \refeq{Fprimebounds} we have that
\begin{align*}
\vert \Delta-  F^\prime(\ust) (u-\ust) \vert
\ge
\vert \Delta \vert -  \vert F^\prime(\ust) (u-\ust) \vert
\ge
\vert \Delta \vert - c_2 \vert \theta \vert \quad \textup{on } [0,1],
\end{align*}
and \refeq{hofmann-estimate} then implies the estimate
\begin{align*}
\vert \Delta \vert - c_2 \vert \theta \vert
\le
\vert \theta \vert \ \vert \Delta \vert
 \quad \textup{on } [0,1].
\end{align*}
For any $ u \in \linf $ satisfying $ \normmax{\theta} \le \myrr < 1 $, we thus have
$ \vert \Delta \vert \le \myrr \vert \Delta \vert + c_2 \vert \theta \vert$
and therefore
$ (1-\myrr)\vert \Delta \vert \le c_2 \vert \theta \vert $
on $ [0,1] $.
This finally yields
$$ \tfrac{1-\myrr}{c_2} \normmax{\Delta} \le \normmax{\theta} \quad \textup{for }
\normmax{\theta} \le \myrr \qquad (0 < \myrr < 1),
$$
from which the first required nonlinearity condition \refeq{normequiv_a} follows immediately.

\item
We next show that \refeq{normequiv_b} holds, in fact for any $ u \in \linf $  sufficiently close to $ \ust $.
From \refeq{Fprimebounds} we have
\begin{align*}
\vert \Delta-  F^\prime(\ust) (u-\ust) \vert
\ge
\vert F^\prime(\ust) (u-\ust) \vert - \vert \Delta \vert
\ge
c_1 \vert \theta \vert - \vert \Delta \vert \quad \textup{on } [0,1],
\end{align*}
and \refeq{hofmann-estimate} then implies that
\begin{align*}
c_1 \vert \theta \vert \le \vert \Delta \vert
 + \vert \theta \vert \, \vert \Delta \vert
 \quad \textup{on } [0,1].
\end{align*}
For any $ 0 < \varepsilon < c_1 $ and $ u \in \linf $ satisfying $ \normmax{\Delta} \le c_1 - \varepsilon $, we thus have
$ c_1\vert \theta \vert \le \vert \Delta \vert + (c_1-\varepsilon) \vert \theta \vert
$
and therefore $ \varepsilon \vert \theta \vert \le \vert \Delta \vert $
on $ [0,1] $.
This provides us with the estimate $ \varepsilon \normmax{\theta} \le \normmax{\Delta},$ which is valid for
$\normmax{\Delta} \le c_1-\varepsilon \; (0 < \varepsilon < c_1).$
This, however, yields directly the second required nonlinearity condition \refeq{normequiv_b}.
\end{myenumerate}
\section{Constructions and verifications}
\label{sec:verifications}
\subsection{Proof of Theorem \ref{th:tifu-well-posed}}
The assertions of this theorem follow, in principle, from standard results on existence and stability of Tikhonov-regularized solutions, which had been presented for example in \cite[Sect.~3]{HKPS07}
and in the monographs \cite[Section~4.1.1]{Schusterbuch12},
\cite[Chapter 3.2]{Scherzetal09},
\cite[Chapter 2.6]{Tikhonov_Leonov_Yagola[98]}, and \cite{Vainikko[88],Vainikko[91.1]}.
Some more details on the applicability of the results in the given references will be given below.  We only note that, due to the compactness of the operator $ G $, we may consider strong topologies.
For convenience of the reader, below we present a detailed proof. We start with the properties of the involved mappings and sets.
\begin{myenumerate_roman}
\item
By assumption, the set $ D(F) $ is a closed subset of $ \ix $.

\item
Also by assumption, the operator $ \F: \ix \supset \Dset \to \yps $ is continuous \wrt the norms on $ \ix $ and $ \yps $.  This implies that the misfit functional
$\myu \in \Dset \mapsto \Jdel{u}^r $ is continuous on $ \ix $ \wrt the norm topology.

\item
We now consider the stabilizing functional $ \Omega: \ix \to [0,\infty] $ given by
$$  \Omega(\myu) = \left\{\begin{array}{ll} \normone{\myu - \ubar}^\myr,
& \textup{if } \myu \in \ix_1, \\
\infty & \textup{otherwise}.
\end{array}\right. $$
We next verify that, for all nonnegative constants $C$, the sublevel sets
$$ \mysl_C \defeq \inset{u \in \ix : \Omega(\myu) \le C} \subset \ix_1  $$
are precompact in $ \ix $ \wrt norm. For this purpose let $ \{\myu_n\}_{n=1}^\infty \subset \mysl_C $ and denote $ v_n \defeq G^{-1}(\myu_n-\ubar) $.
Since the considered space $ \ix $ has a separable predual space $\tilde X$, we may apply the Banach--Alaoglu theorem.
Thus there is a subsequence $\{v_{n_k}\}_{k=1}^\infty$ of the
in $\ix$ bounded sequence $\{v_{n}\}_{n=1}^\infty$, which is weakly* convergent to some element $v_0 \in \ix$.
Since $ G $ is the compact adjoint operator of a bounded linear operator $\tilde G$ in the predual space $\tilde X$,
we obtain the norm convergence
 $ Gv_{n_k} = \myu_{n_k} -\ubar \to u_0:=Gv_0\in \ix$ as $ k \to \infty$,
cf.~Gatica~\cite[Lemma~2.5]{Gatica[18]}.
Hence, we have $ \myu_{n_k} \to \ubar +  u_0 \in \ix $ as $ k \to \infty$.
This shows that $ \mysl_C $ is indeed precompact.

\item
We next show that each sublevel set $ \mysl_C $ is closed in $ \ix $
and that the stabilizing functional $ \Omega $ is lower semicontinuous on $ \mysl_C $, both \wrt the norm topology of $ \ix $.
For this, let $ \{\myu_n\}_{n=1}^\infty \subset \mysl_C $ and $ \myu \in \ix $ with $ \myu_n \to \myu $ as $ n \to \infty $.
This implies that $ G^{-1} \myu_n $ has a weak* convergent subsequence, i.e.,
$ G^{-1} \myu_{n_k} \weakstararrow v \in \ix $ as $ k \to \infty$ for some $ v \in \ix $.
Thus we obtain
strong convergence $ \myu_{n_k} \to Gv \in \ix $ as $ k \to \infty$.
Uniqueness of limits now implies $ u = Gv $, i.e., $ \myu \in \ix_1 = \R(G) $.
A subsequence reasoning shows that
$ G^{-1}\myu_n \weakstararrow G^{-1}\myu \in \ix $ as $ n \to \infty $,
and thus, cf.~\cite[Theorem 9 of Chapter V]{Yosida[80]},
 $ \normone{u-\ubar} \le \liminf_{n\to\infty} \normone{\myu_n-\ubar} $.
This completes the proof of the statement of item (iv).
\end{myenumerate_roman}
We are now in a position to verify the statements (a)--(c) of Theorem
\ref{th:tifu-well-posed}.
\begin{myenumerate}
\item This follows from part (b).

\item
Let $ \{\myu_n\}_{n=1}^\infty \subset \Dset $ be a minimizing sequence for the
\tifu, i.e.,
\begin{align*}
\pad{\myu_n} \to \minpad \defeq \inf_{u \in \Dset} \pad{\myu}.
\end{align*}
This implies
\begin{align*}
\limsup_{n\to\infty} \Omega(\myu_n)
\le \frac{1}{\para} \lim_{n\to\infty} \pad{\myu_n}
= \frac{1}{\para} \minpad,
\end{align*}
and thus $ \sup_{n} \Omega(\myu_n) < \infty $.
From the compactness of the sublevel sets of $ \Omega $ and the closedness of $ \DF $,
norm convergence of some subsequence in $ \Dset $ follows,
i.e., for some $ \udelb \in \Dset $ and some subsequence
$ \{\myu_{n_k}\}_{k=1}^\infty $, we have
strong convergence $ \myu_{n_k} \to \udelb $ as $ k \to \infty $.
The lower semicontinuity of $ \Omega $ and the continuity of $ F $ then imply
\begin{align*}
\myomegabp{\udelb} \le \liminf_{k\to\infty} \myomegabp{\myu_{n_k}},
\qquad
\norm{F(\udelb)-\fdel} = \lim_{k\to\infty} \norm{F(\myu_{n_k})-\fdel},
\end{align*}
and then
\begin{align*}
\pad{\udelb}
&\le
\lim_{k\to\infty} \norm{F(\myu_{n_k})-\fdel}^r + \liminf_{k\to\infty} \myomegab{\myu_{n_k}}
\\
&\le
\liminf_{k\to\infty} \big\{ \norm{F(\myu_{n_k})-\fdel}^r + \myomegab{\myu_{n_k}} \big\}
=
\liminf_{k\to\infty} \pad{\myu_{n_k}} = \minpad
\end{align*}
follows which in fact means
$ \pad{\udelb}  = \minpad $. This completes the proof of part (b) of the theorem.

\item For the verification of stability, consider perturbations of the \tifu of the following form,
\begin{align*}
\padn{u} = \norm{F(u) - \fn}^r + \para \myomegab{u}, \quad u \in \Dset \qquad (n=1,2,\ldots),
\end{align*}
where $ \myseq{f_n} \subset \yps $ with $ \norm{\fn-\fdel} \to 0 $ as $ n \to \infty $.
Let $ \un \in \Dset $ be a minimizer of the \tifu $ \padnpur \ (n = 1,2,\ldots) $ which exists according to part (a) of this theorem.
In what follows, we show that $ \myseq{\un} $ is a minimizing sequence for the original \tifu
$ \padpur $. Stability then follows immediately from part (b) of the theorem.

Let $ \udelb \in \Dset $ be a minimizer of $ \padpur $, and let
\begin{align*}
\hn \defeq \norm{\fn-\fdel}^{\min\{r,1\}}, \quad n = 1,2,\ldots \ .
\end{align*}
Utilizing those notations, we have
\begin{align}
\pad{\un} &=
\norm{F(\un) - \fdel}^r + \para \myomegab{\un}
\nonumber \\
& \le (\norm{F(\un) - \fn} + \norm{\fn -\fdel})^r + \para \myomegab{\un}
\nonumber \\
& \le \norm{F(\un) - \fn}^r + K h_n + \para \myomegab{\un}
= \padn{\un} + K h_n
\label{eq:tifu-well-posed-a} \\
& \le  \padn{\udelb} + K h_n
=
\norm{F(\udelb) - \fn}^r + \para \myomegab{\udelb} + K h_n
\nonumber \\
& \quad  \to
\norm{F(\udelb) - \fdel}^r + \para \myomegab{\udelb}
= \pad{\udelb}
\quad \textup{as} \ n \to \infty,
\nonumber
\end{align}
where $ K \ge 0 $ in \refeq{tifu-well-posed-a}
denotes some finite constant. We note that this inequality
\refeq{tifu-well-posed-a}
follows from the identity
\begin{align*}
(x+h)^r = x^r + \Landauno{h^{\min\{r,1\}}} \quad \textup{ as } h \downarrow 0
\end{align*}
which holds uniformly in $ x \ge 0 $ (on bounded intervals, if $ r \ge 1 $).
Note that
$ \norm{F(\un) - \fn}^r \le \padn{\un} \le \padn{\uhat} $
for any $ \uhat \in \Dset $,
so that
$ \limsup_{n} \norm{F(\un) - \fn}^r \le  \limsup_{n} \padn{\uhat} = \pad{\uhat} $, which implies that the sequence $ \myseq{\norm{F(\un) - \fn}}_n $ is indeed bounded.

We can summarize the above estimate to
\begin{align*}
\limsup_{n\to \infty} \pad{\un}
\le \lim_{n\to \infty} \padn{\udelb}
= \pad{\udelb},
\end{align*}
i.e., $ \myseq{\un} $ is a minimizing sequence for the \tifu $ \padpur $.
This completes the proof of the theorem.
\end{myenumerate}
\proofend
\begin{remark}
From the four items (i)--(iv), the statements (a)--(c)
of Theorem \ref{th:tifu-well-posed} basically follow from standard results on the existence and stability of Tikhonov-regularized solutions given in the references presented in front of the theorem.

\begin{mylist}
\item
For example, parts (a) and (c) are results of Theorems 3.22 and 3.23
in \cite{Scherzetal09}, respectively, if Assumption 3.13 in that reference is considered for norm topologies. The required convexity of the penalty functional is not needed
in our setting. In addition, also the requirement ''exponent $ \ge 1 $" in Assumption 3.13 in \cite{Scherzetal09} may be dropped by noting that Lemma 3.20 in
\cite{Scherzetal09} holds for exponents $ < 1 $, if the constant there is replaced by 1.
Part (b) then is an easy consequence of (c).

\item
Alternatively,  parts (a) and (b) follow directly from
\cite[Lemma 1]{Vainikko[88]}, cf.~also
\cite[Lemma 1]{Vainikko[91.1]},
if one considers the set $ \Dset $ as basic space, equipped with the norm convergence of the space $ \ix $.
Part (c) then is an immediate consequence of (b).
\end{mylist}
\end{remark}
\subsection{Introduction of auxiliary elements}
For the auxiliary elements introduced below, we consider linear bounded regularization operators associated with $ \G $,
\begin{align}
R_\parb: \ix \to \ix	\for \parb > 0
\label{eq:rbeta}
\end{align}
and its companion operators
\begin{align}
S_\parb \defeq  I - R_\parb \G \for \parb > 0.
\label{eq:sbeta}
\end{align}
We assume that the following conditions are satisfied:
\begin{align}
\norm{ R_\parb } & \le \tfrac{c_*}{\parb} \for \parb > 0,
\label{eq:wachstum} \\
\norm{ S_\parb G^\pp } & \le \cp  \parb^{\pp} \for \parb> 0,
\qquad (0 \le \pp \le \pp_0)
\label{eq:abfall} \\
R_\parb G & = G R_\parb \for \parb > 0
\label{eq:commute}
\end{align}
where $ 0 < \pp_0 < \infty $ is a finite number to be specified later, and $ c_* $ and $ \cp $ denote finite constants. We assume that $ \cp $ is bounded as a function of $ p $.
\begin{example} An example is given by Lavrentiev's $m$-times iterated method with an integer $ m \ge 1 $. Here, for $ f \in \ix $ and $ v_0 = 0 \in \ix $, the element $ R_\parb f $ is given by
\begin{align*}
(\G + \parb I)v_n & =  \parb v_{n-1} + f \for n = 1,2,\ldots,m, \qquad
R_\parb f \defeq v_m.
\end{align*}
The operator $ R_\parb $ can be written in the form
$$ R_\parb = \parb^{-1} \sum_{j=1}^{m} \parb^j (G + \parb I)^{-j},  $$
and the
companion operator is given by $ S_\parb = \parb^m (G + \parb I)^{-m} $.
For $ m = 1 $, this gives Lavrentiev's classical regularization method,
$ R_\parb = (G + \parb I)^{-1} $.
For this method, the conditions \refeq{wachstum}--\refeq{commute} are satisfied with
$ p_0 = m $. In fact,
for integer $ 0 \le p \le m $,
estimate \refeq{abfall} holds with constant $ c_p = (\kappa_*+1)^m $, see
\cite[Lemma 1.1.8]{Plato[95]}.
From this intermediate result and the interpolation inequality
\refeq{interpol2}, inequality \refeq{abfall} then follows
for non-integer values $ 0 < p < m $, with constant $ c_p = 2(\kappa_*+1)^{m+1} $.
\remarkend
\end{example}
We are now in a position to introduce \emph{auxiliary elements}
which provide an essential tool for the analysis of the regularization properties of \tikreg considered in our setting.
They are defined as follows,
\begin{align}
\ualpaux \defeq \ubar + R_\parb \G (\ust-\ubar)
=
\ust - S_\parb(\ust-\ubar)
\for \parb > 0,
\label{eq:uaux-def}
\end{align}
where $\G$ is the generator of the scale of normed spaces introduced in Section \ref{postype_operators}, and $ R_\parb, \parb > 0 $, is an arbitrary family of regularizing operators as in \refeq{rbeta} satisfying the conditions \refeq{wachstum}--\refeq{commute}
with saturation
$$ p_0 \ge 1+a, $$
and $ S_\parb, \parb > 0 $, denotes the corresponding companion operators, cf.~\refeq{sbeta}.
In addition,
the solution $ \ust$ of the operator equation \eqref{eq:opeq} and the corresponding initial guess $\ubar $ are as introduced above.
The basic properties of the auxiliary elements \refeq{uaux-def} are summarized in Lemma~\ref{th:auxel} below.

We now state another property of regularization operators which is also needed below.
\begin{lemma}
\label{th:auxel-0}
Let $ R_\parb, \parb > 0 $, be an arbitrary family of regularizing operators as in \refeq{rbeta} satisfying the conditions \refeq{wachstum}--\refeq{commute}.
Then there exist some finite constant $ c > 0  $, so that
for each $ 0 < \pp \le 1 $ we have
\begin{align*}
\norm{ R_\parb \G^\pp } & \le c \parb^{\pp-1} \for \parb > 0.
\end{align*}
\end{lemma}
\begin{proof}
Since $ R_\parb \G^\pp = \G^\pp R_\parb $, for
$ \kappa_1 = 2(\kappa_*+1) $
we have
\begin{align*}
\norm{ R_\parb \G^\pp w } & =
\norm{ G^\pp R_\parb w } \le \kappa_1 \norm{ G R_\parb w }^\pp \norm{ R_\parb w }^{1-\pp}
\\
& \le
\kappa_1(c_0+1)^p c_*^{1-p} \norm{ w } \parb^{\pp-1}, \qquad w \in \ix,
\end{align*}
where the first inequality follows from the interpolation inequality \refeq{interpol2}.
For the meaning of the constants $ c_0 $ and $ c_* $, we refer to
\refeq{wachstum} and \refeq{abfall}, respectively.
\end{proof}
\subsection{Auxiliary results for $ \log \G $}
\begin{lemma}
\label{th:low-order-rate}
For each $ \myu \in \domain(\log \G) $ and each $ 0 \le p < \pp_0 $, we have
$$
\norm{ S_\parb G^p \myu  } =
\Landauno{\beta^p \loginv{\tfrac{1}{\parb}}}
\as \beta \to 0.
$$
\end{lemma}
\begin{proof}
There holds $ \norm{G^q} \le C e^{\omega q} $ for $ q \ge 0 $,
where $ \omega > 0 $ and $ C > 0 $ denote suitable constants, and $ \norm{\cdot} $ denotes the norm of operators on $ \ixtwo $.
This follows, e.g., from the fact
that $ (\G^q)_{q \ge 0} $ defines a $ C_0 $-semigroup on $\ixtwo $.
Thus each real $ \lambda > \omega $ belongs to the resolvent set of the operator $ \log \G: \RG \supset \domain(\log \G) \to \RG $, i.e.,
$ (\lambda I - \log \G)^{-1}: \RG \to \RG $ exists and defines a bounded operator,
cf.~\cite[Theorem 5.3, Chapter 1]{Pazy[83]}.
Since
$$ \R((\lambda I - \log \G)^{-1}) = \domain (\lambda I - \log \G) = \domain( \log \G), $$
we can represent $ u $ as
$$ u = (\lambda I - \log \G)^{-1}w $$
with some $ w \in \RG $. Since (cf.~\cite[proof of Theorem 5.3, Chapter 1]{Pazy[83]})
$$  u = (\lambda I - \log \G)^{-1}w = \int_0^\infty e^{-\lambda q} G^q w \, dq, $$
we have
\begin{align*}
S_\parb G^p u = \int_0^\infty e^{-\lambda q} S_\parb G^{p+q} w \, dq = y_1 + y_2,
\end{align*}
with
\begin{align*}
y_1 = \int_0^{\pp_0-p} e^{-\lambda q} S_\parb G^{p+q} w \, dq,
\qquad y_2 = \int_{\pp_0-p}^\infty e^{-\lambda q} S_\parb G^{p+q} w \, dq.
\end{align*}
The element $ y_1 $ can be estimated as follows for $ \beta < 1 $:
\begin{align*}
\normix{y_1} & \le  c \norm{w} \int_0^{\pp_0-p}  \parb^{p+q}\, dq =
c \norm{w} \parb^p \frac{1}{\log \parb} \parb^q \big\vert_{q=0}^{q=\pp_0-p}
\\
& = c \norm{w}  \parb^p \frac{1}{\vert \log \parb \vert} (1-\parb^{\pp_0-p})
\le c \norm{w}  \parb^p \frac{1}{\vert \log \parb \vert}.
\end{align*}
The element $ y_2 $ can be written as follows,
\begin{align*}
y_2 = \int_{\pp_0-p}^\infty e^{-\lambda q} S_\parb G^{\pp_0} G^{q-(\pp_0-p)} w \, dq,
\end{align*}
and thus we can estimate as follows:
\begin{align*}
\normix{y_2} & \le  c_1 \norm{w} \int_{\pp_0-p}^\infty e^{-\lambda q} \parb^{\pp_0} e^{\omega(q-(\pp_0-p))} \, dq
\\
& \le c_2 \norm{w} e^{-\omega(\pp_0-p)} \parb^{\pp_0}  \int_{\pp_0-p}^\infty e^{-(\lambda-\omega)q} \, dq
=\Landauno{ \parb^{\pp_0}} \as \beta \to 0.
\end{align*}
This completes the proof.
\end{proof}
\begin{lemma}
\label{th:low-order-rate-b}
For each $ \myu \in \domain(\log \G) $, we have
$$
\norm{ R_\parb \myu  } =
\Landauno{\frac{1}{\beta \log \frac{1}{\parb}}}
\as \beta \to 0.
$$
\end{lemma}
\begin{proof}
Follows similar to Lemma \ref{th:low-order-rate}, by making use of
Lemma \ref{th:auxel-0}.
\end{proof}
\subsection{Properties of auxiliary elements}
\label{auxels}
In this section, we present the basic properties of the \emph{auxiliary elements}, which are needed to verify our convergence results.
\begin{lemma}
\label{th:auxel}
Consider the auxiliary elements from \refeq{uaux-def}
with regularization operators $ R_\beta, \beta > 0 $, with saturation
$ \pp_0 \ge 1 + a $.
Let the three function $g_i(\parb)\;(i=1,2,3)$ be given by the following identities:
\begin{align}
& \normix{\ualpaux - \ust}=g_1(\parb),
\label{eq:auxel-a}
\\
& \norma{\ualpaux - \ust} = g_2(\parb)\parb^{a},
\label{eq:auxel-b}
\\
& \normone{\ualpaux-\ubar} = g_3(\parb)\parb^{-1},
\label{eq:auxel-c}
\end{align}
for $ \parb > 0 $, respectively. Those functions $g_i(\parb)\;(i=1,2,3)$  are bounded and have the following properties:
\begin{mylist_indent}
\item (No explicit smoothness) If
$ \ust \in \overline{\R(G)} $, then we have $ g_i(\parb) \to 0 $ as $ \parb \to 0 $ ($i=1,2,3$).
\item (H\"older smoothness)
If $ \ust \in \ix_\pp $ for some $ 0 < \pp \le 1 $,
then
$ g_i(\parb)=\mathcal{O}(\parb^\pp) $ as $ \parb \to 0 $ ($i=1,2,3$),
\item (Low order smoothness)
If $ \ust \in \domain(\log \G) $,
then
$ g_i(\parb)= \Landauno{\loginv{\frac{1}{\parb}}} $
as $ \parb \to 0 $ ($i=1,2,3$).
\end{mylist_indent}
\end{lemma}
\proof
By definition, those three functions $ g_1, g_2 $ and $ g_3 $ under consideration can
be written as follows,
\begin{align*}
g_1(\parb) &= \| S_\parb (\ust-\ubar)\|,
\\
g_2(\parb) &=
\parb^{-a} \| G^a S_\parb (\ust-\ubar) \|, \\
g_3(\parb) &=
\parb \| R_\parb (\ust - \ubar) \|,
\end{align*}
and, according to conditions \refeq{wachstum}--\refeq{commute},
thus are bounded.
\begin{mylist}
\item We consider H\"older smoothness first. Since $ \ust, \, \ubar \in \ix_\pp $ holds, we have $ \ust - \ubar = G^{\pp}w $ for some $ w \in \ix $.
The statements are now easily obtained from \refeq{abfall} and
Lemma~\ref{th:auxel-0}.

\item
We have $ \ust - \ubar \in \domain(\log G) $ and the statements now follow easily from
Lemmas \ref{th:low-order-rate} and \ref{th:low-order-rate-b}.

\item
The convergence statement under any missing smoothness assumption is
based on the \bh{uniform boundedness principle} by taking into account formula
\refeq{abfall} and
Lemma \ref{th:auxel-0}. To apply this \bh{principle}, we consider the
parametric family of linear operators $S_\beta: \ix \to \ix$ and
in this context the associated limiting process $\beta \to 0$ of the
parameter. Then we have uniform boundedness $\|S_\beta\| \le c_0$ for
all $\beta>0$ and convergence
$\|S_\beta z\| \to 0$ as $\beta \to 0$ for all $z$ from the range
$\mathcal{R}(G)$, which is dense in  $\overline{\mathcal{R}(G)}$. \bh{This yields}
$g_1(\beta)=\|S_\beta (u^\dagger-\bar u)\| \to 0$ as $\beta \to
0$ and also the analog assertions for $g_2$ and $g_3$ as required.
\proofend
\end{mylist}
\subsection{Proof of Theorem \ref{th:upardel-esti}}
This section is devoted to the proof of Theorem \ref{th:upardel-esti}. We start with a preparatory lemma.
\begin{lemma}
\label{th:upardel_lemma}
\mainassump
There exists some $ \para_0 > 0 $ such that for $ 0 < \para \le \para_0 $ and each $\delta>0$, we have
\begin{align*}
\max\{\normyps{\myF{\upardel} - \fdel }, \, \para^{1/\myr} \normone{\upardel-\ubar}\}
\le f_2(\para)\para^{\mykap a } + \er \delta.
\end{align*}
Here, $f_2(\para)$ is a bounded function satisfying the following:
\begin{mylist_indent}
\item (No explicit smoothness)
If $ \ust \in \overline{\R(G)} $, then $ f_2(\para) \to 0 $ as $ \para \to 0 $.
\item (H\"older smoothness)
If $ \ust \in \ix_\pp $ for some $ 0 < \pp \le 1 $, then
$ f_2(\para) = \Landauno{\para^{\mykap p}} $ as $ \para \to 0 $.
\item (Low order smoothness)
If $ \ust \in \domain(\log \G) $, then
$ f_2(\para) = \Landauno{\loginv{\frac{1}{\para}}} $ as $ \para \to 0 $.
\end{mylist_indent}
In addition, the constant $\er $ is defined as follows:
$$  \er = \left\{\begin{array}{ll} 1,
& \textup{if } r \ge 1, \\
2^{-1+1/r}& \textup{otherwise}.
\end{array}\right. $$
\end{lemma}
\proof
We consider auxiliary elements of the form \refeq{uaux-def}, 
with saturation $ \pp_0 \ge 1 + a $. We choose
\begin{align}
\parb=\parb(\para) = \para^{\mykap}.
\label{eq:parb-def}
\end{align}
For $ \para > 0 $ small enough, say $ 0 < \para \le \para_0 $, we have
$ \ualpaux \in \Dset $ because of Lemma~\ref{th:auxel} and $\ust$ is an interior point of $\DF$. Thus
we have
\begin{align*}
& \kla{\normypsqua{\myF{\upardel} - \fdel } + \para \normonequa{\upardel-\ubar}}^{1/\myr}
\le \kla{\normypsqua{\myF{\ualpaux} - \fdel } + \para \normonequa{\ualpaux-\ubar}}^{1/\myr}
\\
& \quad \le e_r (\normyps{\myF{\ualpaux} - \fdel } + \para^{1/\myr} \normone{\ualpaux-\ubar}) \\
& \quad
\le e_r(\normyps{\myF{\ualpaux} - \fst } + \para^{1/\myr} \normone{\ualpaux-\ubar} + \delta).
\end{align*}
The first term on the \rhs of the latter estimate can be written as
\begin{align*}
\normyps{\myF{\ualpaux} - \fst }
\le \cb \norma{\ualpaux - \ust}
\le \cb g_2(\parb)\parb^{a}
= \cb g_2(\para^{\mykap})\para^{\mykap a}
\end{align*}
for $ \para $ small enough, say $ \para \le \para_0 $. This is a consequence of estimate \refeq{normequiv_a} and
representation \refeq{auxel-b} of Lemma \ref{th:auxel}.
The second term on the \rhs of the latter estimate attains the form
\begin{align*}
\para^{1/r} \normone{\ualpaux-\ubar}
\le
\para^{1/r} g_3(\parb)\parb^{-1}
= g_3(\para^{\mykap})\para^{\mykap a},
\end{align*}
based on \refeq{auxel-c} of Lemma \ref{th:auxel}. This yields the function
$$f_2(\para):=  \er(\cb g_2(\para^\mykap)+g_3(\para^\mykap)) \quad \textup{for} \;\para \le \para_0. $$
The asymptotical behaviors of the function $ f_2 $ stated in the lemma are immediate consequences of
Lemma~\ref{th:auxel}.
This completes the proof of the lemma.
\proofend
\begin{corollary}
\label{th:upardel_cor}
\mainassump
There exist finite positive constants $ \para_0, \delta_0 $
and $K_2$
such that for $ 0 < \para \le \para_0 $ and each $0 \le \delta \le \delta_0$, we have
$$ \norma{\upardel -\ust} \le f_3(\para)\para^{\mykap} + K_2 \delta.$$
Here $f_3(\para), 0 < \para \le \para_0 $, is a bounded function which satisfies the following:
\begin{mylist_indent}
\item (No explicit smoothness)
If $ \ust \in \overline{\R(G)} $, then $ f_3(\para) \to 0 $ as $ \para \to 0 $.
\item (H\"older smoothness)
If $ \ust \in \ix_\pp $ for some $ 0 < \pp \le 1 $, then
$ f_3(\para) = \Landauno{\para^{\mykap p}} $ as $ \para \to 0 $.
\item (Low order smoothness)
If $ \ust \in \domain(\log \G) $, then
$ f_3(\para) = \Landauno{\loginv{\frac{1}{\para}}} $ as $ \para \to 0 $.
\end{mylist_indent}
\end{corollary}
\proof
Let $ \para $ and $ \delta $ be small enough, say
$ 0 < \para \le \para_0 $
and $ 0 < \delta \le \delta_0 $.
From estimate \refeq{normequiv_b} and Lemma \ref{th:upardel_lemma},
it then follows
\begin{align*}
& \ca \norma{\upardel -\ust} \le \normyps{\myF{\upardel} - \fst } \le
\normyps{\myF{\upardel} - \fdel } + \delta
\le f_2(\para)\,\para^{\mykap a} + (1+\er)\delta.
\end{align*}
The assertion of the corollary now follows by setting $f_3(\para):=\tfrac{f_2(\para)}{\ca}$ and $K_2:=\tfrac{1+\er}{\ca}.$
\proofend
\begin{proof}[Proof of Theorem \ref{th:upardel-esti}]
The error $ \norm{\upardel -\ust} $ is now estimated by the following series of error estimates.
Using $ \parb = \parb(\para) $ from \refeq{parb-def} in combination with
\refeq{auxel-a} from Lemma~\ref{th:auxel}, we obtain
\begin{align}
\normix{\upardel -\ust}\le \normix{\upardel -\ualpaux}+\normix{\ualpaux-\ust}=\normix{\upardel -\ualpaux}+g_1(\para^\mykap),
\label{eq:upardel_prop-a}
\end{align}
and below we consider the term $ \normix{\upardel -\ualpaux} $ in more detail.
From the interpolation inequality \refeq{interpol2}
it follows
\begin{align}
\label{eq:upardel_prop-b}
\normix{\upardel -\ualpaux} & \le
c_3 \norma{\upardel -\ualpaux}^{\lfrac{1}{(a+1)}}
\normone{\upardel -\ualpaux}^{\lfrac{a}{(a+1)}}.
\end{align}
Both terms on the \rhs of the  estimate \eqref{eq:upardel_prop-b} can be estimated by using Corollary~\ref{th:upardel_cor} and Lemma~\ref{th:auxel} in the following manner.
Precisely, we find with
$$f_4(\para):=f_3(\para)+g_2(\para^\mykap), \qquad
f_5(\para):=f_2(\para)+g_1(\para^\mykap)$$
the estimates
\begin{align*}
\norma{\upardel -\ualpaux}
& \le
\norma{\upardel -\ust} + \norma{\ualpaux - \ust} \le f_4(\alpha)\para^{\mykap a}+K_2\delta,\\
\normone{\upardel -\ualpaux}
& \le
\normone{\upardel-\ubar} + \normone{\ualpaux-\ubar}
\le  \para^{-\lfrac{1}{\myr}}\left(f_5(\para)\, \para^{\mykap a}+\delta\right).
\end{align*}
Thus we can continue estimating \refeq{upardel_prop-b}.
Introducing $ f_6(\para):= \max\{ f_4(\para), f_5(\para) \} $
and $ K_3 := \max\{K_2, 1\}, \ K_1 = c_3 K_3 $, we obtain
\begin{align*}
\normix{\upardel -\ualpaux}
& \le c_3 \left(f_4(\alpha)\,\para^{\mykap a } + K_2\delta \right)^{1/(a+1)}\,\left(\para^{-\lfrac{1}{\myr}}\left(f_5(\para)\, \para^{\mykap a}+\delta\right) \right)^{a/(a+1)}
\\
& \le c_3 \left(f_6(\alpha)\,\para^{\mykap a}+K_3\delta \right)^{1/(a+1)}\,\left(\para^{-\lfrac{1}{\myr}}\left(f_6(\para)\, \para^{\mykap a}+ K_3\delta\right) \right)^{a/(a+1)}
\\
& =
c_3 \para^{-\mykap a}
 \left(f_6(\alpha)\,\para^{\mykap a}+K_3\delta \right)
 =
c_3 f_6(\alpha) + K_1 \frac{\delta}{\para^{\mykap a}}.
\end{align*}
From the latter estimate and \refeq{upardel_prop-a}, the theorem now immediately follows by considering $ f_1(\para) := g_1(\para^\mykap) + c_3 f_6(\para) $ there.
\end{proof}
\section*{Acknowledgment}
This paper was created as part of the authors' joint DFG-Project No.~453804957 supported by the German Research Foundation under grants PL 182/8-1 (Robert Plato) and  HO 1454/13-1 (Bernd Hofmann).

\bibliography{oversmooth_banach}
\end{document}